\documentclass[12pt]{amsart}

\usepackage{amsfonts, amsthm, amsmath, amssymb}

\usepackage{amscd}

\usepackage[latin2]{inputenc}

\usepackage{t1enc}

\usepackage[mathscr]{eucal}

\usepackage{indentfirst}

\usepackage{graphicx}

\usepackage{graphics}

\usepackage{pict2e}

\usepackage{mathrsfs}

\usepackage{framed}

\usepackage{mathabx}
\usepackage{booktabs}
\usepackage{makecell}
\usepackage{enumerate}
\usepackage[pagebackref]{hyperref}
\hypersetup{colorlinks=true}
\usepackage{cite}
\usepackage{color}
\usepackage{epic}
\usepackage{hyperref} 
\numberwithin{equation}{section}
\theoremstyle{plain}

\usepackage{tikz}

\usetikzlibrary{calc}
\usetikzlibrary{shapes.geometric}

\tikzset{
  c/.style={every coordinate/.try}
}

\usetikzlibrary{arrows,shapes,positioning}
\usetikzlibrary{decorations.markings}
\tikzstyle arrowstyle=[scale=1]
\tikzstyle directed=[postaction={decorate,decoration={markings,mark=at position 0.6 with {\arrow[arrowstyle]{stealth};}}}]
\tikzstyle reverse directed=[postaction={decorate,decoration={markings,mark=at position 0.4 with {\arrowreversed[arrowstyle]{stealth};}}}]
\tikzstyle dot=[style={circle,inner sep=1pt,fill}]

\newtheorem{theorem}{Theorem}[section]

\newtheorem{lemma}[theorem]{Lemma}

\newtheorem{corollary}[theorem]{Corollary}

\newtheorem{proposition}[theorem]{Proposition}

\theoremstyle{definition}

\newtheorem{example}[theorem]{Example}

\newtheorem{?}[theorem]{Problem}

\newcommand\maxdrop{\operatorname{maxdrop}}

\newcommand\vpk{\operatorname{vpk}}

\newcommand\des{\operatorname{des}}

\newcommand\vnw{\operatorname{vnw}}
\newcommand\inv{\operatorname{inv}}
\newcommand\sor{\operatorname{sor}}
\newcommand\lmax{\operatorname{lmax}}
\newcommand\cyc{\operatorname{cyc}}
\newcommand\DIS{\operatorname{DIS}}
\newcommand\Lmap{\operatorname{Lmap}}
\newcommand\Lmal{\operatorname{Lmal}}
\newcommand\Rmil{\operatorname{Rmil}}
\newcommand\Rmip{\operatorname{Rmip}}
\newcommand\rmin{\operatorname{rmin}}
\newcommand\Cyc{\operatorname{Cyc}}
\newcommand\rmax{\operatorname{rmax}}
\newcommand\Rmal{\operatorname{Rmal}}

\def\boxit#1{\leavevmode\hbox{\vrule\vtop{\vbox{\kern.33333pt\hrule
    \kern1pt\hbox{\kern1pt\vbox{#1}\kern1pt}}\kern1pt\hrule}\vrule}}

\usepackage{collectbox}



\begin{document}

\title[Statistics on permutations with bounded drop size]{Statistics on
permutations with bounded drop size}

\author[J.~N.~Chen]{Joanna N. Chen$^1$}
\address{$^1$ College of Science, Tianjin University of Technology, Tianjin 300384, P.R. China.}
\email{joannachen@tjut.edu.cn}

\author[R.~D.~P.~Zhou]{Robin D.P. Zhou$^{2,*}$}
\address{$^2$ College of Mathematics Physics and Information, Shaoxing University, Shaoxing 312000, P.R. China}
\email{dapao2012@163.com}

\thanks{$^*$ Corresponding author.
Email: dapao2012@163.com.}

\date{\today}

\begin{abstract}
Permutations with bounded drop size, which we also call bounded permutations,  was introduced by Chung, Claesson, Dukes and Graham.  Petersen introduced a new Mahonian statistic  the sorting index, which is denoted by $\sor$.
 Meanwhile, Wilson introduced the statistic $\DIS$, which turns out to   satisfy that
  $\sor(\sigma)=\DIS(\sigma^{-1})$ for any permutation $\sigma$. In this paper, we
maintain Petersen's method to deduce the generating functions of  $(\inv, \lmax)$ and $(\DIS, \cyc)$ over bounded permutations to
 show their equidistribution.  Moreover,  the generating function of  $\des$ over $213$-avoiding bounded permutations and some related equidistributions are given as well.
\end{abstract}


\maketitle


\section{Introduction}\label{sec:intro}
This paper is aim to study statistics on bounded permutations and $213$-avoiding bounded permutations.

Let us give an overview of notation and terminology first.
Denote by $S_n$  the set of the permutations of $[n]$, where $[n]=\{1,2,\ldots,n\}$.
For a permutation $\pi=\pi_1 \pi_2 \cdots \pi_n$ in $S_n$,
a pair $(\pi_i,\pi_j)$ is called an inversion if $i<j$ and $\pi_i>\pi_j$.
Let $\inv(\pi)$ denote the number of inversions of $\pi$.
An element $i$ is called a descent of $\pi$ if $\pi_i >\pi_{i+1}$ and the descent number $\des(\pi)$ is defined to be the number of the descents of $\pi$. We say $i$ is an excedance of $\pi$ if $\pi_i > i$. While, $i$ is called
a weak excedance of $\pi$
if $\pi_i \geq i$.  Otherwise, $i$ is called a non-weak excedance.
Any permutation
statistics with the same distribution as $\inv$ are referred to as Mahonian statistics, while those equidistributed
with  $\des$ and the excedance number are called Eulerian statistics.

Given a word $w=w_1 w_2 \cdots w_n$  of length $n$, let
\begin{align*}
  \Lmap(w) &= \{i\colon w_j <w_i \text{ for all } j<i \}, \\
 \Lmal(w) &=  \{w_i \colon w_j <w_i \text{ for all } j<i \}, \\
  \Rmip(w) &= \{i\colon w_j >w_i \text{ for all } j>i \},\\
  \Rmil(w) &= \{w_i \colon w_j >w_i \text{ for all } j>i \}, \\
  \Rmal(w) &= \{w_i\colon w_j <w_i \text{ for all } j>i \}
\end{align*}
be the set of left-to-right maximum places, left-to-right maximum letters,
right-to-left minimum positions, right-to-left minimum letters  and
right-to left maximum letters respectively.
Their corresponding numerical statistics are denoted by
$\lmax(w) = \# \Lmal(w)$, $\rmin(w) = \#\Rmil(w)$ and
$\rmax(w) = \# \Rmal(w).$

Besides the linear orders, permutations can be also seen as functions.
For example, we may consider $\sigma=34152$ as a function $\sigma \colon [5] \to [5]$ with $
\sigma(1)=3,
\sigma(2)=4, \sigma(3)=1, \sigma(4)=5$ and $\sigma(5)=2$.
Thus, we can write $\sigma=(13)(245)$  and call it a cycle notation of $\sigma$.
Assume that $\sigma=\sigma_1 \sigma_2 \cdots \sigma_n$ and the bijection
$i \mapsto \sigma_i$ has $r$ disjoint cycles, whose minimum elements are $c_1, c_2,\ldots, c_r$.
Let $\Cyc(\sigma)=\{c_1, c_2,\ldots, c_r\}$  and the number of cycles of $\sigma$, $\cyc(\sigma)$, is defined to be the cardinality of $\Cyc(\sigma)$.

An occurrence of a classical pattern $p$ in a permutation $\sigma$ is a subsequence of $\sigma$ that is order-isomorphic to $p$.  Babson and Steingr\'{i}msson \cite{BS} generalized the notion of permutation patterns, to what are now known as vincular patterns, see \cite{Kitaev2011}.
 Adjacent letters in a vincular pattern which are
 underlined must stay adjacent when they are placed  back to the original permutation.
 For an instance, $41253$  contains only one occurrence of the vincular pattern $\underline{31}42$ in its subsequence $4153$, but not in $4253$.

Motivated by juggling sequence and bubble sort, Chung, Claesson, Dukes and Graham \cite{ChungClaesson} introduced the maximum drop size statistic on permutations. For more about bounded permutations, see \cite{ChungGraham, Joanna, Hyatt}.
We say that $\pi \in S_n$ has a drop at $i$ if $\pi_i <i$ and the drop size is meant to be $i-\pi_i$.
The maximum drop size of $\pi$ is defined by
\begin{equation*}
\mathrm{maxdrop(\pi)}=\max\{i-\pi_i\colon 1 \leq i \leq n\}.
\end{equation*}
Let $A(n,k)=\{\pi \in S_n \colon \maxdrop(\pi) \leq k\}$.
For convenience, we call permutations with bounded drop size
the bounded permutations.


Petersen\cite{Petersen} introduced a new  statistic, which is called the sorting index. Given a permutation $\pi$ in $S_n$, it is known that $\pi$
has a unique decomposition into transpositions
\[\pi=(i_1,j_1)(i_2,j_2)\cdots (i_k,j_k),\]
where
\[j_1<j_2< \cdots <j_k\]
and
\[i_1<j_1,i_2<j_2,\ldots,i_k<j_k.\]
The sorting index is defined by
\[\mathrm{sor(\pi)}=\sum_{r=1}^{k}(j_r-i_r).\]

Petersen showed that the sorting index is Mahonian.
Moreover, by using two different factorizations of $\sum_{\pi \in S_n} \pi$, he derived that
\begin{equation*}
    \sum_{\pi \in S_n} q^{\mathrm{sor(\pi)}}t^{\mathrm{cyc(\pi)}}=\sum_{\pi \in S_n} q^{\mathrm{inv(\pi)}}t^{\mathrm{rmin(\pi)}}=\prod_{i=1}^{n}
    (t+[i]_q-1),
\end{equation*}
where $[i]_q=1+q+\cdots+q^{i-1}$.

Based on the standard algorithm for generating a random permutation, Wilson\cite{wilson2010} defined a permutation statistic $\DIS$. It turns out that for a permutation $\pi$ in $S_n$, we have $\DIS(\pi^{-1})=\sor(\pi)$.
The first of our main results in this paper is to generalize Petersen's results to   bounded permutations, which are stated as
follows.

\begin{theorem}\label{thm:tha}
For $0 \leq k<n$, $(\inv,\lmax)$ and $(\DIS,\cyc)$ have the same joint distribution over $A(n,k)$. Moreover,  we have
\begin{equation*}
     \sum_{\sigma \in A(n,k)} q^{\inv(\sigma)}t^{\lmax(\sigma)}=\sum_{\sigma \in A(n,k)} q^{\DIS(\sigma)}t^{\cyc(\sigma)}=(t+[k+1]_q-1)^{n-k}
     \prod_{i=1}^{k}(t+[i]_q-1).
\end{equation*}
\end{theorem}
Moreover, we will show that a bijection constructed by Foata and Han \cite{Han} can be restricted to give an explanation of the equidistribution above.

Hyatt and Remmel \cite{HyattRemmel} initialed
the study of pattern avoiding bounded permutations. They gave the number of $231$-avoiding bounded permutations with $j$ descents.
 Our second main result is concerned with
 the distribution of descents over $213$-avoiding bounded permutations.
Let $A_{n,k}(\sigma)$
 denote the set of $\sigma$-avoiding permutations in $A(n,k)$.
 We have the following proposition.

\begin{proposition}\label{prop:213ballot}
For $0 \leq k \leq n$, we have $A_{n,k}(213)$ is enumerated by the
ballot number $C_{n,k}$, where
\[C_{n,k}=\frac{n-k+1}{n+1}{{n+k}\choose{k}}.\]
\end{proposition}

Let $S_n^k(321)$ be the set of $321$-avoiding permutations of length $n$ with the last element $k+1$. Lin and Kim \cite{LinKim}
mentioned that  $S_{n+1}^k(321)$ is enumerated by $C_{n,k}$.
Bailey \cite{Bailey} introduced non-negative $n,k$-arrangements
of $1$'s and $-1$'s. Assume that $\Gamma_{n,k}$ is the set of sequences with non-negative partial sums that can be formed by
 $n$ $1$'s and $k$ $-1$'s. Namely, given $a=a_1 a_2 \cdots a_{n+k} \in \Gamma_{n,k}$, we have $a_1+\cdots +a_{i} \geq 0$ for $1 \leq i \leq n+k$. Bailey showed that $\Gamma_{n,k}$ is also counted by
 $C_{n,k}$.

Let $\vnw(\pi)$ be the number of the non-weak excedances  that are before the last element of $\pi$.
For $a \in \Gamma_{n,k}$, let $\vpk(a)$ be the number of $a_i=1$ such that $a_i$ is not the rightmost $1$ and $a_{i+1}=-1$ for $1 \leq i \leq n+k-1$.
It should be mentioned that
$a$ can be also seen as a ballot path from $(0,0)$ to $(n,k)$,
while the statistic $\vpk$ can be seen as a variation of  the number of NE-turns over ballot paths as studied by Su \cite{Su}.

We will show a stronger equidistribution
over $A_{n,k}(213)$, $S_{n+1}^k(321)$ and $\Gamma_{n,k}$.

\begin{theorem}\label{thm:equidis}
Statistic $\des$ over $A_{n,k}(213)$,
statistic $\vnw$ over $ S^k_{n+1}(321)$ and statistic $\vpk$ over $ \Gamma_{n,k}$ are equally distributed.
\end{theorem}

\begin{theorem}\label{thm:generating}
Let
\[G(p,q,z)=\sum_{\pi \in A_{n,k}(213),\, 0 \leq k \leq n}p^{\des(\pi)} q^{k}z^n,  \]
then we have
\begin{align}
G(p,q,z)=\frac{1-pqz-(1-z)q+(zq-pqz-q)\widetilde{F}_0}{1-z-pqz-(1-z)q}, \label{eq:main}
\end{align}
where
\[
\widetilde{F}_0=\frac{1-zq(1+p)-\sqrt{(1+zq(1-p))^2-4zq}}{2pqz}.
\]
\end{theorem}

The outline of this paper is as follows. In Section \ref{sec:TH1}, we shall prove Theorem~\ref{thm:tha} algebraically and bijectively.
In Section \ref{sec:TH23}, we will investigate
$A_{n,k}(213)$ as well as the distributions of $\des$
over $A_{n,k}(213)$.

\section{A proof of Theorem~\ref{thm:tha}}\label{sec:TH1}

In this section, we will give two different factorizations of $\sum_{\sigma \in S(n,k)} \sigma$ by maintaining Petersen's method,
where
$$S(n,k)=\{\sigma \in S_n \colon \max_i(\sigma_i-i )\leq k\}.$$
These two factorizations allow us to give a proof of Theorem \ref{thm:tha}.

Recall that both the set of transpositions $T=\{(i,j) \colon 1 \leq i<j \leq n \}$
and the set of adjacent transpositions $S=\{(i,i+1)\colon 1\leq i \leq n-1\}$
are  generating sets for $S_n$. For convenience, we write $t_{i,j}=(i,j)$ and $s_i=(i,i+1)$.  The readers are recommended to see \cite{Brenti} for a comprehensive introduction.

Given a permutation $\pi$ in $S_{n-1}$, in order to generate a permutation in $S_n$, we may insert $n$ in all
possible positions of $\pi$. This can be performed analogously to
generate a permutation in $S(n,k)$ from $S(n-1,k)$.
Clearly, $S(n,k)=S_n$ when $n \leq k$ and $S(n,0)=12 \cdots n$.
Hence, we need only to consider the cases when $0 < k < n$.
In terms of the group algebra, for  $0 < k < n$, we define
\begin{equation*}
\Psi_{i,k}= \left\{
  \begin{array}{ll}
1+s_i + s_i s_{i-1}+ \cdots +s_i s_{i-1} \cdots s_1, & \mbox{ $0<i<k$,}\\[6pt]
1+s_i + s_i s_{i-1}+ \cdots +s_i s_{i-1} \cdots s_{i+1-k}, & \mbox{ $k \leq i \leq n-1$.}
 \end{array} \right.
\end{equation*}
Notice that each element in $\Psi_{i,k}$ can be seen as a reduced expression in $S_{i+1}$.

\begin{lemma}
For $n \geq 2$ and $0< k< n$, we have
\begin{equation}\label{eq:PSI}
  \sum_{w \in S(n,k)}w=\Psi_{1,k}\Psi_{2,k}\cdots\Psi_{n-1,k}.
\end{equation}
\end{lemma}

\begin{proof}
 We will give a proof by induction on $n$. It is clear that $\sum_{w \in S(2,k)} w=\Psi_{1,k}$. We assume that
equation (\ref{eq:PSI}) holds for $n=m-1$, namely,
\begin{equation}\label{eq:induction_m-1}
  \sum_{u \in S(m-1,k)} u=\Psi_{1,k}\Psi_{2,k}\cdots\Psi_{m-2,k}.
\end{equation}
We proceed to prove that it holds for $n=m$.
Clearly, we can identify the elements in $S(m-1,k)$ with the set $$\{\sigma \in S(m,k) \colon \sigma(m)=m\}.$$
Assume that $\sigma=\sigma_1\sigma_2\cdots \sigma_{m-1}m \in  S(m,k)$,
there are two cases to consider.
If $k > m-1$, then we have
\begin{align*}
 \sigma \cdot \Psi_{m-1,k}&= \sigma+ \sigma s_{m-1}+ \cdots +\sigma s_{m-1} s_{m-2} \cdots s_1 \\
  &=\sigma_1\sigma_2\cdots \sigma_{m-1}m+\sigma_1\sigma_2\cdots m\sigma_{m-1}+ \cdots +m\sigma_1\sigma_2\cdots \sigma_{m-1}.
\end{align*}
If $k < m$,  we have
\begin{align*}
 \sigma \cdot \Psi_{m-1,k}&= \sigma+ \sigma s_{m-1}+\sigma s_{m-1} s_{m-2} \cdots +\sigma s_{m-1} s_{m-2} \cdots s_{m-k} \\
  &=\sigma_1\sigma_2\cdots \sigma_{m-1}m+\sigma_1\sigma_2\cdots m\sigma_{m-1}+ \cdots + \sigma_1\sigma_2\cdots\sigma_{m-k-1} m \sigma_{m-k}\cdots  \sigma_{m-1}.
\end{align*}
Note that $m$ appears at
all possible positions in the above two cases such that $max_i\{\pi_i -i \}\leq k$. Hence, we have
\begin{equation}\label{eq:induction_m}
\sum_{w \in S(m,k)} w=\sum_{\sigma \in S(m,k), \sigma(m)=m} \sigma \cdot \Psi_{m-1,k}= \sum_{u \in S(m-1,k)} u \cdot \Psi_{m-1,k}.
\end{equation}
Then, equation (\ref{eq:PSI}) follows from (\ref{eq:induction_m-1})
and (\ref{eq:induction_m}). This completes the proof.
\end{proof}

Based on the factorization above, we obtain the following theorem.
\begin{theorem}\label{tha1}
For $n \geq 1$ and $0< k \leq  n$, we have
\begin{equation}\label{eq:orign}
\sum_{\sigma \in A(n,k)} q^{\mathrm{inv(\sigma)}}t^{\mathrm{lmax(\sigma)}}=(t+[k+1]_q-1)^{n-k}
     \prod_{i=1}^{k}(t+[i]_q-1).
\end{equation}
\end{theorem}

\begin{proof}
It is easy to see that  $\sigma \in A(n,k)$ if and only if
$\sigma^{-1} \in S(n,k)$ and  $lmax(\sigma)=rmin(\sigma^{-1})$.
It follows that
 \[\sum_{\sigma \in A(n,k)} q^{\mathrm{inv(\sigma)}}t^{\mathrm{lmax(\sigma)}}
 =
 \sum_{\sigma \in S(n,k)} q^{\mathrm{inv(\sigma)}}t^{\mathrm{rmin(\sigma)}}.
 \]
 Hence, to prove (\ref{eq:orign}), it suffices to show that
 \begin{equation}\label{eq:suffice}
 \sum_{\sigma \in S(n,k)} q^{\mathrm{inv(\sigma)}}t^{\mathrm{n-rmin(\sigma)}}=(1+t[k+1]_q-t)^{n-k}
     \prod_{i=1}^{k-1}(1+t[i+1]_q-t).
 \end{equation}

Let $\varphi$  be a linear map from $\mathbb{Z}[S(n,k)]$
to $\mathbb{Z}[q,t]$ such that $\mathrm{\varphi(\sigma)=q^{inv(\sigma)}t^{n-rmin(\sigma)}}$.
By the definition of $\Psi_{i,k}$, we see that
\begin{equation}\label{eq:varphipsi}
\varphi(\Psi_{i,k})= \left\{
  \begin{array}{ll}
 1+qt+\cdots +q^i t, & \mbox{ $0<i<k$},\\[6pt]
 1+qt+\cdots +q^k t, & \mbox{ $k \leq i \leq n-1$}.
 \end{array} \right.
\end{equation}

For $u \in S(i,k)$ with $u_i=i$ and $k+1 \leq i \leq n$, we deduce that
\begin{align*}
  \varphi(u \cdot \Psi_{i-1,k}) &= \varphi(u)+ \varphi(u s_{i-1})+\varphi(u s_{i-1} s_{i-2}) + \cdots +\varphi(u s_{i-1} s_{i-2} \cdots s_{i-k})\\[3pt]
   &=\varphi(u)+qt\varphi(u)+q^2t\varphi(u)+
   \cdots+q^kt\varphi(u)\\[3pt]
&=\varphi(u) \varphi(\Psi_{i-1,k}).
\end{align*}
Similarly, we may show that $\varphi(u \cdot \Psi_{i-1,k})=\varphi(u) \varphi(\Psi_{i-1,k})$  for $1 < i \leq k$.

Hence, we deduce that
\begin{align*}
   \varphi(\Psi_{1,k}\Psi_{2,k} \cdots \Psi_{n-1,k})  & =\varphi(\sum_{w \in S(n,k),w_n=n} w \cdot\Psi_{n-1,k} )  \\[3pt]
    &=\sum_{w \in S(n,k),w_n=n} \varphi(w \cdot\Psi_{n-1,k} ) \\[3pt]
    &=\sum_{w \in S(n,k),w_n=n} \varphi(w) \varphi(\Psi_{n-1,k} )\\[3pt]
    &=\varphi(\sum_{w \in S(n,k),w_n=n} w) \varphi(\Psi_{n-1,k} )\\[3pt]
    &=\varphi(\Psi_{1,k}\Psi_{2,k} \cdots \Psi_{n-2,k})\varphi(\Psi_{n-1,k} ).
\end{align*}
Then, (\ref{eq:suffice}) follows from (\ref{eq:varphipsi}).
\end{proof}

By setting $t=1$ in (\ref{eq:orign}), we deduce the following corollary, which coincides with Corollary 2.2
in \cite{ChungClaesson}.
\begin{corollary}
For $n \geq 1$ and $0< k \leq  n$, we have
\begin{equation*}
\sum_{\sigma \in A(n,k)} q^{\mathrm{inv(\sigma)}} =[k+1]_q ^{n-k} [k]_q! .
\end{equation*}
\end{corollary}

Analogously, we proceed to give another factorization of $\sum_{\sigma \in S(n,k)} \sigma$.
For  $0 < k < n$, we define
\begin{equation*}
   \Phi_{i,k}=
   \begin{cases}
  1+t_{i,i+1} + t_{i-1,i+1} + \cdots + t_{1,i+1}; & \mbox{ $0<i < k$}\\
  1+t_{i,i+1} + t_{i-1,i+1} + \cdots + t_{i+1-k,i+1}. & \mbox{ $k \leq i \leq n-1$}\\
   \end{cases}
\end{equation*}

The following lemma shows that how each element in $S(n,k)$
can be generated in terms of the minimal length product of reflections.

\begin{lemma} \label{fac_a2}
For $n \geq 2$ and $0< k< n$, we have
\begin{equation}\label{eq:PHI}
  \sum_{\sigma \in S(n,k)} \sigma=\Phi_{1,k}\Phi_{2,k}\cdots\Phi_{n-1,k}.
\end{equation}
\end{lemma}

\begin{proof}
We will give a proof by induction. Clearly,  $\sum_{\sigma \in S(2,k)} \sigma=\Phi_{1,k}$. We assume that
equation (\ref{eq:PHI}) holds for $n=m-1$.
Now we need to prove that it holds for $n=m$.
Recall that given an element $\sigma=\sigma_1\cdots \sigma_{m-1} $ in $S(m-1,k)$, we may identify  $\sigma$ with the permutation $ \sigma_1 \sigma_2\cdots \sigma_{m-1}m.$ We consider two cases.
If $k > m-1$, then we have
\begin{align*}
 \sigma \cdot \Phi_{m-1,k}&= \sigma+ \sigma t_{m-1,m}+ \cdots +\sigma t_{1,m} \\
  &=\sigma_1\sigma_2\cdots \sigma_{m-1}m+\sigma_1\sigma_2\cdots m\sigma_{m-1}+ \cdots +m\sigma_2\cdots \sigma_{m-1}\sigma_{1}.
\end{align*}
If $ k< m$, then we have
\begin{align*}
 \sigma \cdot \Phi_{m-1,k}&= \sigma+ \sigma t_{m-1,m}+ \cdots +\sigma t_{m-k,m} \\
  &=\sigma_1\sigma_2\cdots \sigma_{m-1}m+\sigma_1\sigma_2\cdots m\sigma_{m-1}+ \cdots + \sigma_1\sigma_2\cdots\sigma_{m-k-1} m \sigma_{m-k+1}\cdots \sigma_{m-1}\sigma_{m-k}.
\end{align*}
Notice that in the second case, the position of $m$ is restricted.
It can only change the position with $\sigma_{m-k},\sigma_{m-k+1}, \ldots, \sigma_{m-1}$.
Thus each element $\pi$ in the above sum satisfy that $max_i\{\pi_i -i \}\leq k$.
Hence, we deduce that
$$
\sum_{w \in S(m,k)} w=\sum_{\sigma \in S(m,k), \sigma(m)=m} \sigma \cdot \Phi_{m-1,k}=\Phi_{1,k}\Phi_{2,k}\cdots\Phi_{m-1,k}.$$
\end{proof}

The factorization above leads to the following theorem.
\begin{theorem} \label{tha2}
For $n \geq 1$ and $0< k \leq  n$, we have
\begin{equation}\label{eq:orign2}
\sum_{\sigma \in A(n,k)} q^{\mathrm{DIS(\sigma)}}t^{\mathrm{cyc(\sigma)}}=(t+[k+1]_q-1)^{n-k}
     \prod_{i=1}^{k}(t+[i]_q-1).
\end{equation}
\end{theorem}

\begin{proof}
Since $\mathrm{sor(\sigma^{-1})=DIS(\sigma)}$  and $\mathrm{cyc(\sigma)=cyc(\sigma^{-1})}$,
it suffices to show that
 \begin{equation}\label{eq:suffice2}
 \sum_{\sigma \in S(n,k)} q^{\mathrm{sor(\sigma)}}t^{n-\mathrm{cyc(\sigma)}}=(1+t[k+1]_q-t)^{n-k}
     \prod_{i=1}^{k-1}(1+t[i+1]_q-t).
 \end{equation}

Let $\phi$  be a linear map from $\mathbb{Z}[S(n,k)]$
to $\mathbb{Z}[q,t]$ such that $\mathrm{\phi(\sigma)=q^{sor(\sigma)}t^{n-cyc(\sigma)}}$.
From the construction, it is easily seen that
\begin{equation*}
  \phi(\Phi_{i,k})=
   \begin{cases}
  1+qt+\cdots +q^i t; & \mbox{ $0<i<k$}\\
  1+qt+\cdots +q^k t. & \mbox{ $k \leq i \leq n-1$}\\
   \end{cases}
\end{equation*}

Let $\sigma=\sigma_1\sigma_2\cdots\sigma_m$ with $\sigma_m=m$.
When exchange the position of $\sigma_m$ and $\sigma_i$, we investigate how the statistics $\mathrm{sor}$ and $\mathrm{cyc}$ behaves in $\sigma$. It is easily seen that
the statistic  $\mathrm{sor}$ is increased by $m-i$.
If $i=m$,
the statistic  $\mathrm{cyc}$ remains the same.
If $i<m$,   $\mathrm{cyc}$ is decreased by $1$.

Assume that $u \in S(i,k)$ with $u_i=i$ and $k+1 \leq i \leq n$. We can verify that
\begin{align*}
  \phi(u \cdot \Phi_{i-1,k}) & =\phi(u)+\phi(ut_{i-1,i})+ \cdots + \phi(ut_{i-k,i}) \\[3pt]
   &=\phi(u)+qt\phi(u)+\cdots+q^kt\phi(u)\\[3pt]
   &=\phi(u) \phi(\Phi_{i-1,k}).
\end{align*}
Similarly, we may prove that $ \phi(u \cdot \Phi_{i-1,k})=\phi(u) \phi(\Phi_{i-1,k})$ for $1<i \leq k$.

Hence, we deduce that $\phi(\Phi_{1,k}\Phi_{2,k} \cdots \Phi_{n-1,k})=\phi(\Phi_{1,k})\phi(\Phi_{2,k}) \cdots \phi(\Phi_{n-1,k})$.
It follows that
\begin{align*}
  \sum_{\sigma \in S(n,k)} q^{\mathrm{sor(\sigma)}}t^{\mathrm{n-cyc(\sigma)}} & =\phi(\sum_{\sigma \in S(n,k)} \sigma) \\[3pt]
  &=\phi(\Phi_{1,k}\Phi_{2,k} \cdots \Phi_{n-1,k}) \\[7pt]
  &=\phi(\Phi_{1,k})\phi(\Phi_{2,k}) \cdots \phi(\Phi_{n-1,k})\\[3pt]
  &=(1+t[k+1]_q-t)^{n-k}
     \prod_{i=1}^{k-1}(1+t[i+1]_q-t).
\end{align*}
This completes the proof.
\end{proof}

Combining Theorem \ref{tha1} and Theorem \ref{tha2}, we
give a proof of Theorem \ref{thm:tha}.

It should be mentioned that Chen, Gong and Guo\cite{Chen}
found that a bijection on $S_n$ defined by Foata and Han\cite{Han} mapped the statistics $\mathrm{(inv,rmin)}$ to
$\mathrm{(sor, cyc)}$. In fact, we claim that this bijection
also serves as a bijection on $S(n,k)$ that sends $\mathrm{(inv,rmin)}$ to $\mathrm{(sor, cyc)}$, which leads to a bijective proof of the equidistribution of $\mathrm{(inv,rmin)}$ and
$\mathrm{(sor, cyc)}$ over $S(n,k)$.
We present Foata and Han's bijection first.

Let $\mathrm{SE_n}$ be the set of all sequences  $a=(a_1,a_2,\ldots,a_n )$ such that $1 \leq a_i \leq i$ for $i \in [n]$. The Lehmer code of a permutation $\sigma=\sigma_1\sigma_2\cdots \sigma_n \in S_n$ is defined to be the sequence $a=Leh(\sigma)=(a_1,a_2,\ldots,a_n )$, where
$$
a_i=|\{j\colon 1 \leq j \leq i, \sigma_j \leq \sigma_i\}|.
$$
Define the A-code of a permutation $\sigma$ by $\mathrm{\text{A-code}(\sigma)=Leh ~\sigma^{-1}}$.
For $\sigma \in S_n$ and $i \in [n]$, let $k =k(i)$ be the smallest integer $k \geq 1$ such that $\sigma^{-k}(i) \leq i$. Then define
$$
\text{B-code}(\sigma)=(b_1,b_2, \ldots, b_n)~~ \text{with} ~~
b_i:=\sigma ^{-k(i)}(i).
$$
Let $\gamma=(\text{B-code})^{-1}\circ \text{A-code}$.

Combining the results of Foata and Han\cite{Han}  and Chen et al.\cite{Chen}, it is easy to see that
\[
(\inv,\Rmil,\Lmap)\sigma=(\sor,\Cyc,\Lmap)\gamma(\sigma).
\]

We assert the following proposition without giving a proof.
\begin{proposition}\label{prop:keeplmax}
 For $\sigma \in S_n$,
\begin{itemize}
  \item [1.] $Lmal(\sigma)=Rmip(\sigma^{-1})=Rmip(Leh(\sigma^{-1}))=Rmip(\text{A-code}(\sigma))$.
  \item [2.]  $Lmal(\sigma)=Rmip(\text{B-code}(\sigma))$.
\end{itemize}
\end{proposition}
Following Proposition \ref{prop:keeplmax}, we deduce that
$Lmal(\sigma)=Lmal(\gamma(\sigma))$.
Notice that the bijection $\gamma$ also keeps $\mathrm{Lmap}$. It follows that $\gamma$ also keeps the
statistic $\mathrm{\max_i \{\sigma_i-i\}}$, which
implies that $\gamma$ is   a bijection on $S(n,k)$.

\begin{proposition}
The map  $\gamma=(\text{B-code})^{-1}\circ \text{A-code}$
keeps the statistic $r:=\mathrm{\max_i \{\sigma_i-i\}}$,
that is, for any permutation $\sigma \in S_n$, we have
\[r (\sigma)=r(\gamma(\sigma)).\]
\end{proposition}
\begin{theorem}
$\gamma$ can be restricted to a bijection over $S(n,k)$ such that for $\sigma \in S(n,k)$
\[(\inv,\Rmil)\sigma=(\sor,\Cyc)\gamma(\sigma).\]
\end{theorem}
\begin{example}
Given $\sigma=571492638$, then $\text{A-code}(\sigma)=123215285$ and
$\text{B-code}^{-1}(123215285)=573291486$.
Hence, we have $\gamma(571492638)=573291486$.
It is easy to check that
\begin{align*}
  \Lmal(571492638)&=\Lmal(573291486)=\{5,7,9\},\\[3pt]
  \Lmap(571492638)&=\Lmap(573291486)=\{1,2,5\},\\[3pt]
  r(571492638)&=r(573291486)=5.
\end{align*}
\end{example}

\section{ $A_{n,k}(213)$ and Distribution of $\des$ over  $A_{n,k}(213)$ }\label{sec:TH23}

In this section, we will show that $A_{n,k}(213)$ is calculated by
the ballot number. Further, we provide bijections between any two of the structures $A_{n,k}(213)$, $S_{n+1}^{k}(321)$ and $\Gamma_{n,k}$.
Based on this, we deduce the distribution of $\des$ over $A_{n,k}(213)$,
which can be seen as  refinements of both Narayana numbers and ballot numbers.

To give a proof of Proposition \ref{prop:213ballot},
we first recall some properties of $C_{n,k}$.
The beginnings of $C_{n,k}$ are given in Table 1.

\begin{table}[htp]\label{tab:ballot}
\begin{center}
\begin{tabular}{ c| c c c c c c c c }
  \hline
  {n}$\backslash${k} & 0 & 1 & 2 & 3 & 4 & 5 &6 &7\\ \hline
  0 & 1 &   &   &   &   &  & & \\
  1 & 1 & 1 &   &   &   &  & & \\
  2 & 1 & 2 & 2 &   &   &  & & \\
  3 & 1 & 3 & 5 & 5 &   &  & & \\
  4 & 1 & 4 & 9 & 14 & 14 & & &  \\
  5 & 1 & 5 & 14 & 28 & 42 & 42 &  &  \\
  6 & 1 & 6 & 20 & 48 & 90 & 132 & 132 &  \\
  7 & 1 & 7 & 27 & 75 & 165 & 297 & 429 &429  \\
  \hline
\end{tabular}
\end{center}
\vspace{0.5cm}
\caption{The beginnings of the ballot number $C_{n,k}$}
\end{table}
It can be checked that $C_{n,k}$ satisfies the following recurrences
\begin{eqnarray}
  C_{n,k} &=& C_{n-1,k}+C_{n,k-1}, \label{eq:re1}\\[4pt]
  C_{n,k}&=& C_{n-1,k}+C_{n-1,k-1}+\cdots + C_{n-1,1}+C_{n-1,0}. \label{eq:re2}
\end{eqnarray}

\begin{lemma}\label{lem:lastmaxdrop}
For $\pi \in S_n(213)$, we have
$\maxdrop(\pi)=n-\pi_n$.
\end{lemma}

\begin{proof}
Assume that $\pi_n=i$, we have
$1,2, \ldots, i$ is a subsequence of $\pi$.
It follows that $\pi^{-1}(j)\leq n-i+j$ for $1 \leq j \leq i-1$, namely,  $\pi^{-1}(j)-j \leq n-i$. Moreover, it is easy to check that  $ \pi^{-1}(j)-j < n-i$ for $ i<j \leq n$, as desired.
\end{proof}

Now we are ready to give a proof of Proposition \ref{prop:213ballot}.
\begin{proof}[Proof of Proposition~\ref{prop:213ballot}]
Assume that
$A^i_{n}(213)=\{ \pi \in A_{n}(213) \colon \pi_n=n-i\}$. Then, we have
\[A_{n,k}(213)=  \bigcup_{i=0}^k  A^i_{n}(213). \]

We aim to show that  $\# A^i_{n}(213) = \# A_{n-1,i}(213).$
Based on Lemma \ref{lem:lastmaxdrop},
this can be verified through a simple bijection from $A^i_{n}(213)$
to $A_{n-1,i}(213)$ by deleting the last element and getting the standard order. Hence, $\#A_{n,k}(213)=  \Sigma_{i=0}^k  \# A_{n-1,i}(213)$. In view of equation (\ref{eq:re2}), we deduce that $\#A_{n,k}(213)=  C_{n,k}$.
\end{proof}

\begin{corollary}
Sets  $A_{n,k}(132)$, $A_{n,k}(2\underline{13})$ and
$A_{n,k}(\underline{13}2)$ are all enumerated by $C_{n,k}$.
\end{corollary}

\begin{proof}
Since $\pi^{rc} \in S_{n,k}(132)$  if and only if
$\pi  \in A_{n,k}(213)$, then  $\#A_{n,k}(213)=\#A_{n,k}(132)$.
Notice  that $\pi$ is $213$-avoiding if and only if $\pi$
is $2\underline{13}$-avoiding, see Lemma 2.8 in \cite{Fu}. Hence, $\#A_{n,k}(213)=\#A_{n,k}(2\underline{13})$. Similarly,
we have $\#A_{n,k}(132)=\#A_{n,k}(\underline{13}2)$, as desired.
\end{proof}

As mentioned by Lin and Kim \cite{LinKim},
 $S_{n+1}^k(321)$ is enumerated by $C_{n,k}$.
 We will deduce a recurrence  of $\#S_{n}^k(321)$, which
proves this fact again.
\begin{proposition}\label{prop:321}
For $0 \leq k \leq n$, we have $\#S_{n+1}^k(321)=C_{n,k}$.
\end{proposition}

\begin{proof}
We divide $S_n^k(321)$ into two parts, namely,
\[S_n^k(321)=\{\pi \in S_n^k(321)  \colon k \geq \pi^{-1}(k) \} \cup \{\pi \in S_n^k(321)  \colon k < \pi^{-1}(k) \}.\]
It suffices to show that
\begin{equation}\label{eq:321a}
  \#\{\pi \in S_n^k(321)  \colon k \geq \pi^{-1}(k) \}=C_{n-1,k-1},
\end{equation}
and
\begin{equation}\label{eq:321b}
  \#\{\pi \in S_n^k(321)  \colon k < \pi^{-1}(k) \}=C_{n-2,k}.
\end{equation}

To show (\ref{eq:321a}), it is enough to construct a bijection
from $\{\pi \in S_n^k(321)  \colon k \geq \pi^{-1}(k) \}$
to $\{\pi \in S_{n}^{k-1}(321) \}$ by exchanging $k$ and $k+1$ in a permutation.

To show (\ref{eq:321b}), we need to construct a bijection from
$\{\pi \in S_n^k(321)  \colon k < \pi^{-1}(k) \}$
to $\{\pi \in S_{n-1}^k(321)\}$. Given $\pi \in S_n^k(321)$
with $k < \pi^{-1}(k)$, it is easy to see that $\pi_{n-1}=k$
or $\pi_{n-1}=n$.
If $\pi_{n-1}=n$, then delete $n$.
If $\pi_{n-1}=k$,  then delete $k$ and change $k+2$ to $k$ and
$i$ to $i-1$ for $k+3 \leq i \leq n$. It is easy to check that this map is a bijection and we omit it here.
In view of (\ref{eq:re1}), we complete the proof.
\end{proof}

In the following, we will construct bijections from
$A_{n,k}(213)$ to $\Gamma_{n,k}$ and  from
$S^k_{n+1}(321)$ to $\Gamma_{n,k}$ to give a proof of
Theorem \ref{thm:equidis}.
\begin{proof}[Proof of Theorem \ref{thm:equidis}]
Firstly, we
construct a bijection $\alpha$ from  $A_{n,k}(213)$ to $\Gamma_{n,k}$.

Given $\pi=\pi_1 \pi_2 \cdots \pi_n \in A_{n,k}(213)$,
let the right-to-left maximum letters of $\pi$ be $\pi_{i_1}, \pi_{i_2},
\ldots,\pi_{i_r}$ with $i_1 <i_2 < \cdots <i_r$ and $\pi_{i_1}
>\pi_{i_2}> \cdots > \pi_{i_r}$.
Create an non-negative $n,k$-arrangement $a=\alpha(\pi)$ by the following steps
\begin{itemize}
  \item Firstly, writing $i_1$ $1$'s.
  \item Secondly, for $j$ from $1$ to $r-1$, adjoining $(\pi_{i_{j}}-\pi_{i_{j+1}})$ $-1$'s with
$(i_{j+1}-i_{j})$ $1$'s .
  \item Lastly, writing $(\pi_n-n+k)$ $-1$'s.
\end{itemize}

It is easy to check that $a$ has $i_r$ $1$'s and $(\pi_{i_1}-n+k)$ $-1$'s.
Since $i_r=n$ and $\pi_{i_1}=n$, then $a$ is an $n,k$-arrangement.
Notice that,  for $1 \leq j \leq r-1$, elements larger than $\pi_{i_{j+1}}$ is at the position
$i_j$ or to the left. Hence, $i_j \geq n- \pi_{i_{j+1}}$, which implies that $a$ is non-negative.

To show that $\alpha$ is a bijection, we shall construct its inverse.
Let $a \in \Gamma_{n,k}$  with $\vpk(a)=r-1$ and $r \geq 1$.
Assume that $a$ has $p_i$ $1$'s and $n_i$ $-1$'s alternatively for
$1 \leq i \leq r$, then the corresponding $\pi$ in $A_{n,k}(213)$
can be constructed as follows.

The positions of the right to left maxima of $\pi$  are
$n-p_r- p_{r-1}-\cdots -p_2$,  $n-p_r-p_{r-1}-\cdots-p_3$, $\cdots$, $n-p_r$, $n$.
The values of the right to left maxima of $\pi$ are
$n-k+n_{r}+n_{r-1}+\cdots +n_1$, $n-k+n_{r}+n_{r-1}+\cdots +n_2$, $\cdots$, $n-k+n_r$.
Then  picking out the largest $p_r-1$ remaining elements that are smaller than $n-k+n_r$ and filling them in the positions between $n-p_r$ and $n$ in increasing order. Similarly,
we can fill the positions between $n-p_r-\cdots-p_{j+1}-p_j$
and $n-p_r-\cdots-p_{j+1}$ in the same way for $1<j<r$.
From the construction of $\pi$, it is easily to check that
$\pi$ is $132$-avoiding.
By Lemma \ref{lem:lastmaxdrop}, we obtain that
$\maxdrop(\pi) = n - \pi_n = n-(n-k+n_r)=k-n_r\leq k$.
Hence $\pi$ is a permutation in $A_{n,k}(213)$ and
$\alpha$ is a bijection between $A_{n,k}(213)$ and $\Gamma_{n,k}$.

Moreover, it is easy to see from the construction of the bijection $\alpha$ that $\des(\pi)=\rmax(\pi)-1=\vpk(a)$.

Now, we proceed to give a bijection $\beta$ from $S_{n+1}^k(321)$ to $\Gamma_{n,k}$.
Notice  that
 a permutation is $321$-avoiding if and only if both the subsequence formed by its weak excedance values and the one formed
by the remaining non-weak excedance values are increasing.
Given

$\pi=\pi_1 \pi_2 \cdots \pi_{n+1}
 \in S_{n+1}^k(321)$,
let the non-weak excedance values
of $\pi$ that are before position $n+1$ be $\pi_{i_1}, \pi_{i_2},
\ldots,\pi_{i_r}$ with $i_1 <i_2 < \cdots <i_r$ and $\pi_{i_1}
<\pi_{i_2}< \cdots < \pi_{i_r}$.
Create a non-negative $n,k$-arrangement $a=\beta(\pi)$ by the following steps
\begin{itemize}
  \item Firstly, writing $(i_1-1)$ $1$'s and $\pi_{i_1}$ $-1$'s.
  \item Secondly, for $j$ from $1$ to $r-1$, adjoining $(i_{j+1}-i_{j})$ $1$'s with $(\pi_{i_{j+1}}-\pi_{i_{j}})$ $-1$'s.
  \item Lastly, writing $(n+1-i_{r})$ $1$'s and $(k-\pi_{i_r})$ $-1$'s.
\end{itemize}

 Since $i_j-1 \geq \pi_{i_j}$ and $n \geq k$, it is
easy to check that $a$ is an non-negative $n,k$-arrangement.

To show $\beta$ is a bijection,  we will construct its inverse.
Let $a \in \Gamma_{n,k}$ with $\vpk(a)=r-1$
and $p_i$ $1$'s and $n_i$ $-1$'s alternatively for
$1 \leq i \leq r$.
 We construct  the corresponding $\pi$ in $S^k_{n+1}(321)$ as follows.

Elements $k-n_{r}- \cdots -n_3-n_2 $,  $k-n_{r}- \cdots -n_3 $, $\cdots$, $k-n_{r}$ are at the positions
$p_1+1$, $p_1+p_2+1$, $\cdots$, $p_1+p_2+\cdots +p_{r-1}+1$.
Element $k+1$
is at the position $n+1$. The remaining elements are placed
in the remaining places in increasing order.

By the definition of $\beta$, we may easily deduce that  $\vpk(a)=\vnw(\pi)$. The proof is completed.
\end{proof}

\begin{example} Let $n=8$, $k=7$ and
 $\pi=83475612$, the right to left maxima of $\pi$
are $\pi_1=8, \pi_4=7, \pi_6=6, \pi_8=2$.
Then, we have \[a=\alpha(\pi)=1,-1,1,1,1,-1,1,1,-1,-1,-1,-1,1,1,-1.\]

As the inverse of $\alpha$, for
 $a=1,-1,1,1,1,-1,1,1,-1,-1,-1,-1,1,1,-1$,
we have $n=8$, $k=7$, $r=\vpk(a)+1=4$, $p_1=1$, $p_2=3$, $p_3=2$, $p_4=2$,
$n_1=1$, $n_2=1$, $n_3=4$ and $n_4=1$.
Then $8,7,6,2$ are placed at positions $1,4,6,8$
and $\pi=\alpha^{-1}(a)=83475612$.

Let $n=8$, $k=7$ and $\sigma=314527698$. The non-weak excedance values of $\sigma$ that are before position $9$
are $\sigma_2=1, \sigma_5=2, \sigma_7=6$.
Then, we have \[b=\beta(\sigma)=1,-1,1,1,1,-1,1,1,-1,-1,-1,-1,1,1,-1.\]

As the inverse of $\beta$, for $b=1,-1,1,1,1,-1,1,1,-1,-1,-1,-1,1,1,-1$,
we have $n=8$, $k=7$, $r=\vpk(b)+1=4$, $p_1=1$, $p_2=3$, $p_3=2$, $p_4=2$,
$n_1=1$, $n_2=1$, $n_3=4$ and $n_4=1$.
Then $1,2,6$ are placed at positions $2,5,7$
and $\sigma=\beta^{-1}(b)=314527698$.
\end{example}

For  completeness, we present the bijection from $A_{n,k}(213)$ to $S_{n+1}^k(321)$ as follows.

Given $\pi=\pi_1 \pi_2 \cdots \pi_n \in A_{n,k}(213)$,
let the right-to-left maximum letters of $\pi$ be $\pi_{i_1}, \pi_{i_2},
\ldots,\pi_{i_r}$ with $i_1 <i_2 < \cdots <i_r$ and $\pi_{i_1}
>\pi_{i_2}> \cdots > \pi_{i_r}$.
Construct a permutation $\sigma \in S^k_{n+1}(321)$ by the following
steps:

\begin{itemize}
  \item The positions of the non-weak excedances of $\sigma$ before the last element  are $i_1+1$, $i_2+1$, $\ldots$, $i_{r-1}+1$.
  \item  The values of the the non-weak excedances of $\sigma$ before the last element  are $n-\pi_{i_2}$, $n-\pi_{i_3}$, $\ldots$,
      $n-\pi_{i_r}$.
  \item  Place $k+1$ in the $n+1$-th place and then put the remaining elements to the remaining places in increasing order.
\end{itemize}

It is easy to check that it is a bijection and the inverse of the bijection is omitted. As an example, $83475612$ is mapped to $314527698$.

In the following, we will give the distribution of $\des$
over $A_{n,k}(213)$. By Theorem \ref{thm:equidis},
it suffices to deduce the distribution of $\vpk$ over $\Gamma_{n,k}$.

\begin{proof}[Proof of Theorem \ref{thm:generating}]
Let
\[F(p,q,z)=\sum_{a \in \Gamma_{n,k},\, 0 \leq k \leq n
} p^{\vpk(a)} q^{n-k}z^n,\]

\[F^u(p,q,z)=\sum_{a \in \Gamma^u_{n,k},\, 0 \leq k \leq n
} p^{\vpk(a)} q^{n-k}z^n,\]

\[F^d(p,q,z)=\sum_{a \in \Gamma^d_{n,k},\, 0 \leq k \leq n
} p^{\vpk(a)} q^{n-k}z^n.\]
where $\Gamma^u_{n,k}$ ($\Gamma^d_{n,k}$) is the set of  arrangements $a \in \Gamma(n,k)$ with  $a_{n+k}=1$ ($a_{n+k}=-1$).

Let
\begin{equation*}
  F(p,q,z)=\sum_{k \geq 0} F_k(p,z)q^k = \sum_{n \geq 0} f_n(p,q)z^n,
\end{equation*}
\begin{equation*}
  F^u(p,q,z)=\sum_{k \geq 0} F^u_k(p,z)q^k = \sum_{n \geq 0} f_n^u(p,q)z^n,
\end{equation*}
\begin{equation*}
  F^d(z,p,q)=\sum_{k \geq 0} F^d_k(p,z)q^k = \sum_{n \geq 0} f^d_n(p,q)z^n.
\end{equation*}
Then we have $f^u_1(p,q)=q$.  By considering two cases of the
the last element of an non-negative $n,k$-arrangement,
being $1$ or $-1$,  we deduce the following two recurrences.
For $n \geq 1$, we have
\begin{equation}\label{eq:furecurrence}
  f^u_{n+1}(p,q) =q f^u_{n}(p,q)  + pq f^d_{n}(p,q).
\end{equation}
For $n \geq 0$, we have

\begin{equation}\label{eq:fdrecurrence}
    f^d_{n+1}(p,q) =q^{-1} f^u_{n+1}(p,q)  + \{q^{ \geq 0}\}(q^{-1} f^d_{n+1}(p,q)),
\end{equation}
where $\{q^{ \geq 0}\}$ is the linear operator extracting all terms
in the power series representation containing non-negative powers of $q$.

Summing (\ref{eq:furecurrence}) over $n \geq 1$ and summing
(\ref{eq:fdrecurrence}) over $n \geq 0$, we deduce that
\begin{align*}
  \sum_{n \geq 1} f^u_{n+1}(p,q) z^{n+1}&=\sum_{n \geq 1} (q f^u_{n}(p,q)  + pq f^d_{n}(p,q)) z^{n+1}, \\[4pt]
  \sum_{n \geq 0} f^d_{n+1}(p,q) z^{n+1}&=\sum_{n \geq 0} (q^{-1} f^u_{n+1}(p,q)  + \{q^{ \geq 0}\}(q^{-1} f^d_{n+1}(p,q))) z^{n+1}.
\end{align*}
We deduce that
\begin{align}
  &F^u(p,q,z)-qz=qz F^u(p,q,z)+pqzF^d(p,q,z) , \label{eq:genre1} \\[5pt]
 &F^d(p,q,z)=q^{-1}F^u(p,q,z)+q^{-1}F^d(p,q,z)-q^{-1}F_0(p,z), \label{eq:genre2}
\end{align}
where
\[F_0(p,z)=\sum_{a \in \Gamma_{n,n},\, n \geq 1} p^{\vpk(a)} z^n.\]

Recall that the Narayana number $N_{n,m}$ is the number of Dyck path
of semilength $n$ having $m$ peaks and
\begin{equation*}
  \sum_{1 \leq m \leq n} N_{n,m} x^n y^m= \frac{1-x(1+y)-\sqrt{(1+x(1-y))^2-4x}}{2x},
\end{equation*}
see \cite{Callan} for reference.

\begin{align}\label{eq:F0}
 F_0(p,z)&=\frac{\sum_{1 \leq m \leq n} N_{n,m} z^n p^m}{p} \notag\\[4pt]
 &=\frac{1-z(1+p)-\sqrt{(1+z(1-p))^2-4z}}{2zp}.
\end{align}

By solving (\ref{eq:genre1}) and (\ref{eq:genre2}), we derive that
\begin{align*}
  F^u(p,q,z) & = \frac{(1-q^{-1})qz-pzF_0(z,p)}{1-qz-pz-(1-qz)q^{-1}},    \\[6pt]
  F^d(p,q,z) & = \frac{z-q^{-1}(1-qz)F_0(z,p)}{1-qz-pz-(1-qz)q^{-1}}.
\end{align*}

Then
\begin{align*}
  F(p,q,z) & =1 + F^u(z,p,q)+ F^d(z,p,q) \\[5pt]
   & =\frac{1-pz-(1-qz)q^{-1}+(z-pz-q^{-1})F_0(z,p)}{1-qz-pz-(1-qz)q^{-1}}.
\end{align*}
Hence, by changing $q$ to $q^{-1}$ and $z$ to $zq$ above we
obtain (\ref{eq:main}).
\end{proof}

\section*{Acknowledgement}
The first author was supported by the National Natural Science Foundation of China (No.~11701420) and
Natural Science Foundation Project of Tianjin Municipal Education Committee (No.~2017KJ243,~No.~2018KJ193).
The second author was supported by the National Natural Science Foundation of China (No.~11801378,~No.~12071440) and the Zhejiang Provincial Natural Science Foundation of China (No.~LQ17A010004).


\begin{thebibliography}{99}
\bibitem{BS}E.~Babson and E.~Steingr\'{i}msson, Generalized permutation patterns and a classification of the Mahonian statistics, S\'{e}m. Lothar. Combin,  B44b  (2000): 18~pp.

\bibitem{Bailey}
D. F. Bailey, counting arrangements of 1's and -1's,
 Mathematics Magazine 69 (1996), 128--131.

\bibitem{Brenti}
A. Bj\"{o}rner and F. Brenti, Combinatorics of Coxeter groups, Springer-Verlag, New York, 2005.

\bibitem{Callan}
D. Callan,  T. Mansour and M. Shattuck, Restricted ascent sequences and Catalan numbers, Appl. Anal. Discr. Math. 8 (2014), 288--303.



\bibitem{Joanna}
J. N. Chen and W.Y.C. Chen,  On permutations with bounded permutations, European J. Combin. 54 (2016) 138--153.


\bibitem{Chen}
W.Y.C. Chen, G.Z. Gong, J.J.F. Guo,  The sorting index and permutation coding, Adv. App. Math. 50 (2013) 367--389.

\bibitem{ChungClaesson}
 F. Chung, A. Claesson, M. Dukes and R. Graham, Descent polynomials for permutation with bounded drop size, European J. Combin. 31 (2010) 1853--1867.

\bibitem{ChungGraham}
F. Chung and R. Graham, Inversion-descent polynomials for restricted permutations, J. Combin. Theory Ser. A 120 (2013) 366--378.


\bibitem{Han}
 D. Foata, G.N. Han, New permutation coding and equidistribution of set valued statistics, Theoret. Comput. Sci. 410 (2009) 3743--3750.


 \bibitem{Fu}
S. Fu,  D. Tang, B. Han and J. Zeng, $(q,t)$-Catalan numbers: gamma expansions, pattern avoidances, and the $(-1)$-phenomenon,
Adv. App. Math. 106 (2019) 57--95.



\bibitem{Hyatt}
 M. Hyatt, Descent polynomials for $k$ bubble-sortable permutations of type $B$, European J. Combin. 34 (2013) 1171--1191.

\bibitem{HyattRemmel}
M. Hyatt and J. Remmel, The classification of $231$-avoiding permutations by descents and maximum drop, J. Combin. 6 (2015) 509 -- 534.


\bibitem{Kitaev2011}
 S. Kitaev, Patterns in permutations and words, Springer Science \& Business Media, 2011.

 \bibitem{LinKim}
 Z. Lin and D. Kim, Refined restricted inversion sequences,
 arXiv:1706.07208v2.

\bibitem{Petersen}
T. K. Petersen, The sorting index, Adv. App. Math. 47 (2011) 615--630.


\bibitem{Su}
X. Su, Enumerating closed flows on forks, Discrete Math. 340 (2017) 3002--3010.


\bibitem{wilson2010}
Mark C. Wilson, An interesting new Mahonian permutation statistic, Electron. J. Combin. 17 (2010), R147.




\end{thebibliography}
\end{document}